\documentclass[12pt]{amsart} 
\usepackage{amsfonts,amsthm,amsmath}
\usepackage{rotating}
\usepackage{tikz}
\usepackage{verbatim}
\usepackage{amscd}
\usepackage[latin2]{inputenc}
\usepackage{t1enc}
\usepackage[mathscr]{eucal}
\usepackage{indentfirst}
\usepackage{graphicx}
\usepackage{graphics}
\usepackage{pict2e}
\usepackage{mathrsfs}
\usepackage{enumerate}
\usepackage[pagebackref]{hyperref}
\usepackage{cite}
\usepackage{epic}
\usepackage{hyperref} 
\usepackage{tkz-graph}

\numberwithin{equation}{section}

\theoremstyle{plain}
\newtheorem{theorem}{Theorem}[section]
\newtheorem{lemma}[theorem]{Lemma}

\newtheorem{Observation}[theorem]{Observation}

\theoremstyle{definition}
\newtheorem{Def}[theorem]{Definition}
\newtheorem{example}[theorem]{Example}

\newtheorem{remark}[theorem]{Remark}
\newtheorem{?}[theorem]{Problem}

\newcommand{\N}{\mathbb{N}}

\newcommand{\Z}{\mathbb{Z}}

\newcommand\A{\mathsf{A}}
\renewcommand\P{\mathsf{P}}
\newcommand\D{\mathsf{D}}

\def\plat{\mathsf{plat}}
\def\hill{\mathsf{hill}}
\def\rmin{\mathsf{rmin}}
\def\lmax{\mathsf{lmax}}

\def\rdes{\mathsf{rdes}}
\def\dez{\mathsf{dez}}
\def\reduction{\mathrm{red}}
\def\xdes{\mathsf{xdes}}
\def\Der{\mathrm{Der}}
\def\DES{\hbox{\rm\textsc{des}}}
\def\DEZ{\hbox{\rm\textsc{dez}}}
\def\DEC{\hbox{\rm\textsc{dec}}}
\def\INC{\hbox{\rm\textsc{inc}}}
\def\REC{\hbox{\rm\textsc{rec}}}
\def\NIW{\hbox{\rm\textsc{niw}}}
\def\Pos{\mathrm{Pos}}

\def\id{\mathrm{id}}
\def\ID{\mathrm{ID}}
\def\Fix{\hbox{\rm\textsc{fix}}}
\def\fix{\mathsf{fix}}
\def\lseg{\mathsf{lseg}}
\def\seg{\mathsf{seg}}

\def\max{\mathrm{max}}
\def\min{\mathrm{min}}

\def\exc{\mathsf{exc}}
\def\des{\mathsf{des}}

\def\aexc{\mathsf{aexc}}

\def\S{\mathfrak{S}}

\def\W{\mathcal{W}}

\def\eff{\mathrm{eff}}

\def\ldes{\mathsf{ldes}}

\def\a{\mathfrak{a}}
\def\b{\mathfrak{b}}

\def\ASC{\hbox{\rm\textsc{asc}}}

\def\max{\operatorname{max}}

\def\s{\operatorname{{\bf s}}}
\def\df{\operatorname{df}}
\def\pf{\operatorname{pf}}

\def\s{\operatorname{{\bf s}}}

\def\U{\underline}
\def\B{\bf}

\def\ee{\mathbf{e}}

\def\neg#1{{{{\bar#1}}}}

\def\boxit#1{\leavevmode\hbox{\vrule\vtop{\vbox{\kern.33333pt\hrule\kern1pt\hbox{\kern1pt\vbox{#1}\kern1pt}}\kern1pt\hrule}\vrule}}

\usepackage{collectbox}

\begin{document}

\title[$k$-arrangements]{ $k$-arrangements, statistics and  patterns}

\author[S. Fu]{Shishuo Fu}
\address[Shishuo Fu]{College of Mathematics and Statistics, Chongqing University, Huxi campus, Chongqing 401331, P.R. China}
\email{fsshuo@cqu.edu.cn}

\author[G.-N. Han]{Guo-Niu Han}
\address[Guo-Niu Han]{I.R.M.A., UMR 7501, Universit\'e de Strasbourg et CNRS, 7 rue Ren\'e Descartes, F-67084 Strasbourg, France}
\email{guoniu.han@unistra.fr}

\author[Z. Lin]{Zhicong Lin}
\address[Zhicong Lin]{Research Center for Mathematics and Interdisciplinary Sciences, Shandong University, Qingdao 266237, P.R. China}
\email{linz@sdu.edu.cn}
\date{\today}

\begin{abstract} The $k$-arrangements  are permutations whose fixed points are $k$-colored. We prove enumerative results related to statistics and patterns on $k$-arrangements, confirming  several conjectures by Blitvi\'c and Steingr\'imsson.  In particular, one of their conjectures regarding the equdistribution of the number of descents over the derangement form and the permutation form of $k$-arrangements is strengthened in two interesting ways. Moreover, as one application of the so-called Decrease Value Theorem, we calculate the generating function for a symmetric pair of Eulerian statistics over permutations arising in our study. 
\end{abstract}


\maketitle


\section{Introduction}\label{sec:intro} 

In their course~\cite{bs} of interpreting moments of probability measures on the real line, Blitvi\'c and Steingr\'imsson introduced the {\em$k$-arrangements}, which are permutations with $k$-colored fixed points. They posed several  conjectures related to the equidistributions of statistics and enumeration of patterns on $k$-arrangements. The purpose of this note is to address these enumeration conjectures.
Let $\S_n$ be the set of all permutations of $[n]:=\{1,2,\ldots,n\}$. 
For each permutation 
$\sigma=\sigma(1)\sigma(2)\cdots\sigma(n) \in\S_n$ 
let $\Fix(\sigma):=\{i\in[n]:\sigma(i)=i\}$ 
be the set of {\em fixed points} of $\sigma$. 
For any nonnegative integer $k$, a {\em$k$-arrangement} of $[n]$ is a 
pair $\mathfrak{a}=(\pi,\phi)$ of a permutation $\pi\in\S_n$ 
and an arbitrary function
$\phi:\Fix(\pi)\rightarrow\{\neg i: 1\leq i\leq k\}$, 
where $\neg i:=-i$. Note that for $k=0$,  there is no function $\phi:\Fix(\pi)\rightarrow\emptyset$ unless $\Fix(\pi)=\emptyset$.  We will refer to $\pi$ as the {\em base permutation} of $\a$. Let $\A_n^k$ denote the set of $k$-arrangements of $[n]$.  
For instance, the $0$-arrangements and $1$-arrangements can be identified with derangements and permutations, respectively. 
The $2$-arrangements, also called  {\it decorated permutations} by Postnikov~\cite[Def.~13.3]{po}, were investigated previously from different aspects~\cite{cor,lwz,po,wi}.  

Blitvi\'c and Steingr\'imsson~\cite{bs} introduced two different representations of $k$-arrangements, called {\em permutation form} and {\em derangement form}. 
Define the {\em reduction} (resp. {\em positive reduction})  of a word $w$ over integers, denoted by $\reduction(w)$ (resp. $\reduction^+(w)$),  to be the word obtained from $w$ by replacing all instances of 
the $i$-th smallest letter (resp. positive letter) of $w$ with~$i$, for all $i$. For example, we have $\reduction(55\neg12\neg1\neg2)=442321$
and $\reduction^+(55\neg12\neg1\neg2)=22\neg11\neg1\neg2$.
For a $k$-arrangement $\a=(\pi,\phi)$ of $[n]$, the {\em derangement form} (resp. {\em permutation form}) of $\a$, denoted $\df_k(\a)$ (resp. $\pf_k(\a)$),  is the word obtained from $\pi$ by changing $\pi(i)$ to $\phi(i)$ 
for each $i\in\Fix(\pi)$  (resp. $i\in\Fix(\pi)$ such that $\phi(i)\neq \neg k$) and then applying the positive reduction. 
For instance,  let $\a$ be the $3$-arrangement $(\pi, \phi)$ with $\pi=75{\bf3}{\bf4}1{\bf6}2$ and $\phi(3)=\neg 1$, $\phi(4)=\neg 3$ and $\phi(6)=\neg 3$. 
Then $\a$ has derangement form $43\neg1\neg31\neg32$ whose {\em derangement part} is $\Der(\a)=4312$, and 
permutation form $64\neg13152$ whose {\em permutation part} is $643152$. 
The set of permutation forms (resp.~derangement forms) representing elements in $\A_n^k$ is denoted $\P_n^k$ (resp.~$\D_n^k$). Note that $\P_n^1=\S_n$.
The above two representations of $k$-arrangements provide	
two bijections between these three sets:
	$$\pf_k: \A_n^k\rightarrow \P_n^k \quad \text{and} \quad \df_k: \A_n^k\rightarrow \D_n^k.$$

For a word $w=w_1w_2\cdots w_n$ over $\Z$, an index $i\in[n-1]$ is called a 
{\it descent (position)} of $w$ if $w_i>w_{i+1}$. 
Let $\DES(w)$ (resp. $\des(w)$) be the set (resp. the number) of descents of a word $w$, and $\Pos(w)$ be the {\em positive subword} of $w$, i.e., the subword that is consisted of all the positive letters in $w$. It is well known (cf.~\cite{fs}) that the {\em Eulerian polynomials} $A_n(t)$ can be defined by
$ A_n(t)=\sum_{\pi\in\S_n}t^{\des(\pi)} $.
We extend some classical statistics on 
permutations or words
to  $k$-arrangements $\a=(\pi, \phi)$ by 
\begin{align*}
\fix(\a)&:=|\Fix(\pi)|, & \fix_i(\a)&:=|\{j\in\Fix(\pi):\phi(j)=\neg i\}|, \\
\DES(\a)&:=\DES(\pf_k(\a)), &     \des(\a)&:=\des(\pf_k(\a)),\\
\DEZ(\a)&:=\DES(\df_k(\a)),  &    \dez(\a)&:=\des(\df_k(\a)),\\
\Der(\a)&:=\Pos(\df_k(\a)).  &
\end{align*}
We split $\A_n^k$ into small subsets according to the multiplicity of the image of the funciton $\phi$ in the following way.
Given an array of nonnegative integers ${\bf m}=(m_1, m_2, \cdots, m_k)$,
we let 
$$
\A_n^k({\bf m}) := \{
	\a \in \A_n^k : \fix_i(\a)=m_i  \text{\ for $1\leq i\leq k$}
\}.
$$
In the previous work of Foata and the second author \cite{fh1,fh3},  derangement forms in $\D_n^1$ were already studied under the term {\em shuffle class}.
They also introduced the $\DEZ$ and $\dez$ statistics on permutations,
and constructed the bijection $\Phi=\Phi_1$ to show the following equidistribution, in the case of $k=1$.

\begin{theorem}\label{euler:k-set}
Let $n,k\geq1$ and 
${\bf m}=(m_1, m_2, \cdots, m_k)$ be an array of nonnegative integers.
For any permutation $\tau\in \S_k$,
	there exists a bijection $\Phi_k: \A_n^k({\bf m}) \rightarrow \A_n^k({\bf m'})$ such that for every $\a\in\A_n^k(\bf m)$,
\begin{equation}\label{eq:equieuler}
(\DEZ,  \Der)~\a = (\DES,  \Der)~\Phi_k(\a),
\end{equation}
	where ${\bf m'} = (m_{\tau(1)}, m_{\tau(2)}, \ldots, m_{\tau(k)})$.
\end{theorem}
The equidistribution of $\des$ over
$\P_n^k$ and $\D_n^k$, 
was first conjectured in~\cite[Conj.~1]{bs}, 
which is generalized in two directions:
a set-valued extension, as stated in
Theorem \ref{euler:k-set},  
and a symmetrical
generalization as stated in the next theorem.
Notice that Theorem \ref{double-euler}
is new even for $k=1$ over $\S_n$, since our triple
equidistribution on $\S_n$ can not be proven using the bijection in~\cite[Thm.~1.1]{fh1}.

\begin{theorem}\label{double-euler}
Let $n,k\geq1$ and 
${\bf m}=(m_1, m_2, \cdots, m_k)$ be an array of nonnegative integers.
	For any permutation $\tau\in \S_{k-1}$,
	there exists a bijection $\Psi_k: \A_n^k({\bf m}) \rightarrow \A_n^k({\bf m'})$ such that for every $\a\in\A_n^k(\bf m)$,
\begin{equation}\label{eq:double-euler}
(\des, \dez, \Der)~\a = (\dez, \des, \Der)~\Psi_k(\a).
\end{equation}
	where ${\bf m'} = (m_{\tau(1)}, m_{\tau(2)}, \ldots, m_{\tau(k-1)}, m_k)$.
\end{theorem}

Taking the identity permutation $\tau$, Theorems \ref{euler:k-set} and \ref{double-euler} imply that there exists two bijections $\a \mapsto \a'$ and $\a\mapsto \a''$  from $\A_n^k$ onto itself such that
\begin{align*}
	(\DEZ,  \Der)~\a &= (\DES,  \Der)~\a',\\
	(\des, \dez, \Der)~\a &= (\dez, \des, \Der)~\a''.
\end{align*}

\medskip

We will also investigate the enumerative aspect of pattern avoiding $k$-arrangements.
We say a word $w=w_1w_2\cdots w_n\in\Z^n$ {\em avoids the pattern} $\sigma=\sigma_1\sigma_2\cdots \sigma_k\in\S_k$ ($k\leq n$) if there does not exist $i_1<i_2<\cdots<i_k$ such that  $\reduction(w_{i_1}w_{i_2}\cdots w_{i_k})=\sigma$. For a set $\mathcal{W}$ of words, let $\mathcal{W}(\sigma)$ be the set of $\sigma$-avoiding words in $\mathcal{W}$. Two patterns $\sigma$ and $\pi$ are said to be {\em Wilf-equivalent} over $\mathcal{W}$ if $|\mathcal{W}(\sigma)|=|\mathcal{W}(\pi)|$. 
One of the most famous enumerative results in pattern avoiding permutations, attributed to MacMahon and Knuth (cf.~\cite{ki,ss}), is that $|\S_n(\sigma)|=C(n)$ for each pattern $\sigma\in\S_3$, where 
$$C(n):=\frac{1}{n+1}\binom{2n}{n}
$$ is the $n$-th {\em Catalan number}. The study of pattern avoiding derangements was  initiated by Robertson,  Saracino and  Zeilberger~\cite{rsz}  and further generalized by others in~\cite{eli,ep,ck}. Since the $k$-arrangements in permutation form and derangement form can be considered as generalizations of permutations and derangements, we study $k$-arrangements of both forms avoiding a single  pattern of length $3$. In the case of permutation form, we verify all the enumerative conjectures (see Section \ref{sec:5}) posed by Blitvi\'c and Steingr\'imsson~\cite{bs}, while in the derangement form only one Wilf-equivalence is found
and reproduced next.
\begin{theorem}\label{der:wilf}
For $n\geq1$, we have $|\D_n^1(321)|=|\D_n^1(132)|$. In other words, the pattern $321$
 is Wilf-equivalent to $132$ on $1$-arrangements in derangement form. Moreover, we have the algebraic generating function for $|\D_n^1(321)|$:
 \begin{equation}\label{gen:321der}
 1+\sum_{n\geq1}|\D_n^1(321)|x^n=\frac{1-3x+3x^2+2x^3+(x^2+x-1)\sqrt{1-4x}}{2x^2(1-x)(2+x)}.
 \end{equation}
 \end{theorem}

The rest of this paper is organized as follows. In Section~\ref{sec:3}, we provide explicit bijections $\Phi_k$ and $\Psi_k$ for proving Theorems~\ref{euler:k-set} and \ref{double-euler}. In Section~\ref{sec:4}, using the so-called Decrease Value Theorem, we calculate the generating function for a symmetric pair of Eulerian statistics $\des$ and $\dez$ over permutations. The enumeration of pattern avoiding $k$-arrangements are carried out in Section~\ref{sec:5}, proving all the connections suspected by Blitvi\'c and Steingr\'imsson, as well as 
Theorem \ref{der:wilf}.


%
\section{Constructions of the bijections \texorpdfstring{$\Phi_k$ and $\Psi_k$}{Phi-k and Psi-k}}\label{sec:desdez:sym}
\label{sec:3} 
In this section we describe explicit constructions of bijections $\Phi_k$ and $\Psi_k$ mentioned in Section~\ref{sec:intro}, and then prove Theorems~\ref{euler:k-set} and \ref{double-euler}.
Foata and the second author have constructed a $\DES$-preserving bijection $\Phi$ between $\D_n^1$ and $\P_n^1$ in a different form~(see \cite[Thm.~1.1]{fh1}). 
The reader is referred to \cite{fh1} for the definition and properties of $\Phi$. Our general bijection $\Phi_k$ is constructed by using $\Phi$ composed with other simple transformations, including the following multiplicity changing bijection $\theta$ 
\cite[Section 4]{han1992}. 
\begin{lemma}\label{lemma:mul0}
Let  
${\bf m}=(m_1, m_2, \cdots, m_k)$ be an array of nonnegative integers
and 
${\bf n}=(n_1, n_2, \cdots, n_k)$ a rearrangement of ${\bf m}$.
There exists a bijection 
	$\theta: R({\bf m}) \rightarrow R({\bf n})$ such that for each $w \in R({\bf m})$,
$$
	\DES(w) = \DES(\theta(w)),
	$$
	where $R({\bf m})$ (resp. $R({\bf n})$) denotes the set of all words on $k$ linearly ordered letters $a_1<a_2<\cdots<a_k$, containing exactly $m_i$ (resp. $n_i$) copies of the letter $a_i$ for all $i=1,2,\ldots, k$.
\end{lemma}

\begin{lemma}\label{lemma:mul1}
Let $n,k\geq1$ and 
${\bf m}=(m_1, m_2, \cdots, m_k)$ be an array of nonnegative integers.
For any permutation $\tau\in \S_k$,
	there exists a bijection 
	$\a \in \A_n^k({\bf m}) \mapsto \b \in \A_n^k({\bf m'})$ such that 
\begin{equation}
(\DEZ,  \Der)~\a = (\DEZ,  \Der)~\b,
\end{equation}
where ${\bf m'} = (m_{\tau(1)}, m_{\tau(2)}, \ldots, m_{\tau(k)})$.
\end{lemma}
\begin{proof}
	(Step 1)	In the derangement form $\df_k(\a)$ of $\a$, hide the positive letters; (Step 2) Then, apply the appropriate $\theta$ in Lemma \ref{lemma:mul0}; (Step 3) Show the letters hidden in (Step 1). We get the derangement form 
	$\df_k(\b)$ of $\b$. 
	There are three types of descent values in $\df_k(\a)$ and $\df_k(\b)$ : positive letter to positive letter, negative letter to negative letter, positive letter to negative letter. Checking each type of descent values, we conclude that $\DES(\df_k(\a))=\DES(\df_k(\b))$.
\end{proof}
\begin{lemma}\label{lemma:mul2}
Let $n,k\geq1$ and 
${\bf m}=(m_1, m_2, \cdots, m_k)$ be an array of nonnegative integers.
	For any permutation $\tau\in \S_{k-1}$,
	there exists a bijection $\a\in \A_n^k({\bf m}) \mapsto \b\in \A_n^k({\bf m'})$ such that 
\begin{equation}
(\DES, \DEZ, \Der)~\a = (\DES, \DEZ, \Der)~\b.
\end{equation}
	where ${\bf m'} = (m_{\tau(1)}, m_{\tau(2)}, \ldots, m_{\tau(k-1)}, m_k)$.
\end{lemma}
\begin{proof}
	Similar to the proof of Lemma \ref{lemma:mul1}, we construct $\b$ and verify that $\DES(\df_k(\a))=\DES(\df_k(\b))$. 
	Since all letters $\neg k$ are not changed in (Step 2), we also have $\DES(\pf_k(\a))=\DES(\pf_k(\b))$.
\end{proof}

\begin{proof}[Proof of Theorem \ref{euler:k-set}]
By Lemma \ref{lemma:mul1}, it suffices to prove the theorem for a special permutation $\tau\in\S_k$.
We then prove the theorem for $\tau=\tau(1)\tau(2)\cdots\tau(k-1)\tau(k)=2 3 \cdots k 1\in \S_k$.
The bijection $\a \in \A_n^k({\bf m}) \mapsto \Phi_k(\a) \in \A_n^k({\bf m'})$ is constructed  in the following way.

Step 1. Derive the derangement form $S_1 = \df_k(\a)\in \D_n^k$ of $\a$;

Step 2. From $S_1$, replace each $-j$ by $-j+1$;  hide all negative letter; 

Step 3. Apply the bijection $\Phi$ described in \cite[Section 2]{fh1} and obtain $S_3=\Phi(S_2)$;

Step 4. From $S_3$, show the letters hidden in Step 2; replace $0$ by $\neg k$.
We get $S_4 \in \D_n^k$;

Step 5. Finally let $\Phi_k(\a)=\df_k^{-1}(S_4)$. For convenience, we write $S_5=\pf_k(\Phi_k(\a))\in \P_n^k$.

\noindent
The following example illustrates our construction by showing the result of each step.
\medskip
$$
\begin{array}{c c c  c c c c c c c c c c c c c c c c c c c c c}
	&S_1 &=&5&\bar 1&1&\bar 2&2 &\bar3    &\bar 1&\bar 3&\bar 1&3&6&\bar 2&\bar 1&7&4 & \in \D_n^k\\
	&S_2 &=&5&0&1&&2   & &0&&0&3&6&&0&7&4 &\\
	&S_3 &=&5&1&0&&0   & &2&&3&0&6&&0&7&4 & \\
	&S_4 &=&5&1&\bar 3&\bar 1&\bar 3   & \bar2&2&\bar 2&3&\bar 3&6&\bar 1&\bar 3&7&4 & \in \D_n^k\\
	&	S_5 &=&8&1& 3&\bar 1 &4   & \bar 2 &2&\bar2 &5&7&10&\bar1&9&11&6 & \in \P_n^k\\
\end{array}
$$
With the bijection $\Phi_k$ constructed above, we verify easily 
\begin{equation*}
	\Der~(\a) = \Der~(\Phi_k(\a)),
\end{equation*}
and
	\begin{equation}\label{eq:mulchange}
	(\fix_1, \fix_2, \fix_3, \ldots, \fix_k)~\a 
	= (\fix_k,\fix_1,\fix_2, \ldots, \fix_{k-1} )~\Phi_k(\a).
\end{equation}
Moreover, by the construction and properties of $\Phi$ (see \cite[Thm.~1.1]{fh1}), we have
\begin{equation*}
	\DEZ(\a) = \DES(S_1)=  \DES(S_5)= \DES(\Phi_k(\a)).
\end{equation*}
In the above example, 
$	\DEZ(\a) = \{1,3,5,7,11,14\}= \DES(\Phi_k(\a))$.
This proves Theorem~\ref{euler:k-set}.
\end{proof}

We need some definitions to facilitate our construction of $\Psi_k$. 
Let us denote by 
$\Der_k(\a)$
the word obtained from $\df_k(\a)$ by removing all letters $\neg k$, called the  {\em weak derangement part} of $\a$, which is extremely  important in our construction. Note that $\Der_k(\a)$ can be viewed as the derangement form of certain $k$-arrangement itself, hence $(\pi,\phi)=\df_k^{-1}(\Der_k(\a))$ is well-defined with $\phi^{-1}(\neg k)=\emptyset$. For each permutation $\sigma=\sigma(1)\cdots\sigma(n)\in\S_n$, we define its {\em excedance word} $\ee(\sigma)$ to be the word made from two letters $E$ and $N$, standing for {\em excedance} and {\em nonexcedance}, respectively. More precisely, we let
$$\ee(\sigma)=e_1e_2\cdots e_n, \quad \text{where } e_i:=\begin{cases}
E & \text{if $\sigma(i)>i$,}\\
N & \text{if $\sigma(i)\le i$,}
\end{cases}\quad \text{for }1\le i\le n.$$

The excedance word for $\Der_k(\a)$, is understood to be the excedance word for the base permutation of $\df_k^{-1}(\Der_k(\a))$. A moment of reflection should reveal the following observation, which shows that inserting letters $\neg k$ back into $\Der_k(\a)$ does not change the excedance type of those letters contained in $\Der_k(\a)$.
\begin{Observation}\label{exc word}
For all $n,k\ge 1$ and each $\a=(\pi,\phi)\in\A_n^k$, suppose $\Der_k(\a):=w=w_1\cdots w_m\in\Z^m$ for some $m\le n$, and $w_j$ is reduced from $\pi(i_j)$ for every $1\le j\le m$. If $\ee(\pi)=e_1e_2\cdots e_n$ and $\ee(w)=e_1'e_2'\cdots e_m'$ are the excedance words for $\pi$ and $w$ respectively, then $e_{i_j}=e_j'$ for every $1\le j\le m$.
\end{Observation}

For the $3$-arrangement $\a=(\pi,\phi)$ given in the introduction, namely, $\pi=75{\bf3}{\bf4}1{\bf6}2$ and $\phi(3)=\neg 1$, $\phi(4)=\neg 3$ and $\phi(6)=\neg 3$, we see $\Der_3(\a)=43\neg112$. Therefore $\ee(\pi)=EENNNNN$ and $\ee(\Der_3(\a))=EENNN$, which agrees with the observation above.

Suppose $\a=(\pi,\phi)\in\A_n^k$ with $\fix_k(\a)=n-m$, $\Der_k(\a):=w=w_1 \cdots w_{m}\in\Z^m$, and $\ee(w)=e_1\cdots e_m$ for some $1\le m\le n$. Now $\df_k(\a)$ can be decomposed as 
$$\df_k(\a)=S_0\:w_1\:S_1\:w_2\:\cdots\:S_{m-1}\:w_m\:S_m,$$
where $S_i$, $0\le i\le m$ is a (possibly empty) block of letters $\neg k$, referred to as the {\em $i$-th slot} of $\df_k(\a)$. Define the {\em slot length vector} of $\a$ as $\s(\a):=(s_0,s_1,\ldots,s_{m})$, where $s_i=|S_i|$ for $0\le i\le m$. Note that $n-m=\sum_{i=0}^{m}s_i$ and the pair $(\s(\a),\Der_k(\a))$ uniquely determines $\a$ and vice versa. 

Next, we set $w_0=w_{m+1}=+\infty$, $e_0=e_{m+1}=E$, and classify $S_i$ into the following mutually exclusive types, according to the values of $w_i$, $w_{i+1}$, and the pair $(e_i,e_{i+1})$:
\begin{itemize}
	\item type I: $w_{i} > w_{i+1}$ and $(e_i,e_{i+1})=(E,N)$;
	\item type II: $w_{i} \le w_{i+1}$ and $(e_i,e_{i+1})=(N,E)$;
	\item type III: $w_{i} \le w_{i+1}$ and $(e_i,e_{i+1})\not=(N,E)$;
	\item type IV: $w_{i} > w_{i+1}$ and $(e_i,e_{i+1})\not=(E,N)$.
\end{itemize}

The four types above clearly cover all the possibilities for the slot $S_i$, and by Observation~\ref{exc word}, the type of $S_i$ only depends on $\Der_k(w)$ and has nothing to do with $s_i$. We use $t_1(\a)$ (resp.~$t_2(\a)$, $t_3(\a)$ and $t_4(\a)$) to denote the number of slots (possibly empty) of type I (resp.~type II, type III and type IV) in $\df_k(\a)$, while the numbers of the non-empty ones are denoted as $t_1^{+}(\a)$, $t_2^{+}(\a)$, $t_3^{+}(\a)$ and $t_4^{+}(\a)$ respectively. We can easily verify that
\begin{align*}
&t_1(\a)+t_4(\a) = \des(\Der_k(\a))+1,\\
&t_1(\a)+t_2(\a)+t_3(\a)+t_4(\a) = m+1, \text{ and }\\
&t_1^+(\a)+t_2^+(\a)+t_3^+(\a)+t_4^+(\a) \le n-m.
\end{align*}

The following lemma is the key to motivate our definition of $\Psi_k$.
\begin{lemma}\label{block type}
For every $\a\in\A_n^k$ with a non-empty $\Der_k(\a)$, we have
\begin{itemize}
	\item[1)] $t_1(\a)=t_2(\a)$ and slots of type I and type II appear alternatingly in $\df_k(\a)$, starting with a block of type I. 
	\item[2)] the following relationships hold
	\begin{align}
	\label{des-change}\des(\a) &=\des(\Der_k(\a))+t_1^{+}(\a)+t_3^{+}(\a),\\
	\label{dez-change}\dez(\a) &=\des(\Der_k(\a))+t_2^{+}(\a)+t_3^{+}(\a).
	\end{align}
\end{itemize}
\end{lemma}
\begin{proof}
Suppose $w=\Der_k(\a)$ as before. It is evident that 
\begin{itemize}
\item A type I slot cannot precede another type I slot unless it precedes a type II slot first. 
\item A type II slot cannot precede another type II slot unless it precedes a type I slot first.
\item A type IV slot $S_i$ with $w_i>w_{i+1}$ and $e_i=e_{i+1}=E$ cannot precede a type II slot unless it precedes a type I slot first.
\end{itemize}
Since we made the convention that $w_0=w_{m+1}=+\infty$, and that the positive letters of $w$ form a derangement word, we see $S_m$ must be of type II, and $S_0$ is either of type IV with $w_0>w_1>0$, $e_0=e_1=E$, or of type I when $w_1$ is negative. Hence by the discussion above, there exists at least one slot of type I among $S_0,S_1,\ldots,S_{m-1}$, and the claim in part 1) follows as well.

Next for part 2), when $\fix_k(\a)=0$ and $m=n$, we have $\df_k(\a)=\pf_k(\a)=\Der_k(\a)$, with all slots being empty, so $\des(\a)=\dez(\a)=\des(\Der_k(\a))$. Otherwise, we can recover $\df_k(\a)$ from $\Der_k(\a)$ by inserting $n-m$ copies of $\neg k$ into originally empty slots. Note that inserting $j$ copies of $\neg k$ into certain slot $S_i$ has the same effect on $\des$ (resp. $\dez$) for each $1\le j\le n-m$. Suppose $S_i$ is an empty slot of type I, i.e., $w_{i} > w_{i+1}$ and $(e_i,e_{i+1})=(E,N)$. Now if we insert $j$ copies of $\neg k$ into it, transfering
$$w_1\cdots w_i w_{i+1}\cdots w_m \quad \text{into} \quad
	w_1\cdots w_i\underbrace{\bar k \bar k\cdots \bar k}_{\text{$j$ copies}} w_{i+1}\cdots w_m,
$$
we see $\des$ increases by one while $\dez$ remains the same. This explains why we have the term $t_1^+(\a)$ in \eqref{des-change} but not in \eqref{dez-change}. Similar discussions of the other three types prove both \eqref{des-change} and \eqref{dez-change}.
\end{proof}

\begin{proof}[\bf Proof of Theorem~\ref{double-euler}]
By Lemma \ref{lemma:mul2}, it suffices to prove the theorem for a special permutation $\tau\in\S_{k-1}$. We then prove the theorem for the identity permutation $\tau=12\cdots (k-1)\in \S_{k-1}$. We proceed to construct a bijection $\a\in\A_n^k({\bf m})\mapsto \Psi_k(\a)\in\A_n^k({\bf m})$ such that
\begin{align}\label{eq:Psi-Der_k}
(\des,\dez,\Der_k)~\a &= (\dez,\des,\Der_k)~\Psi_k(\a),
\end{align}
which is more than we need, since $\Der(\a)=\Pos(\Der_k(\a))$.
The only $k$-arrangement with empty weak derangement part is $\a=(12\cdots n,\phi)$, where $\phi(i)=\neg k$ for all $1\le i\le n$. In this case we let $\Psi_k(\a)=\a$ and see that \eqref{eq:Psi-Der_k} holds true. Otherwise, for a given $k$-arrangement $\a\in\A_n^k$ with
$$\s(\a)=(s_0,s_1,\ldots,s_m), \; \text{ and }\; \Der_k(\a)\in\Z^m, \; m\ge 1.$$
The aforementioned map $\Psi_k$ simply swaps the slots of types I and II in $\df_k(\a)$. Namely, by Lemma~\ref{block type}, we let $t:=t_1(\a)=t_2(\a)$, and let $S_{i_1},S_{j_1},\ldots,S_{i_t},S_{j_t}=S_m$ be all of the types I and II blocks, appearing alternatingly. Then $\b:=\Psi_k(\a)$ is taken to be the unique $k$-arrangement corresponding to
$$\s(\b)=(s_0',s_1',\ldots,s_m'),\; \text{ and }\; \Der_k(\b)=\Der_k(\a),$$
where
$$s_l':=\begin{cases}
s_{j_r} & \text{if $l=i_r$, for certain $1\le r\le t$},\\
s_{i_r} & \text{if $l=j_r$, for certain $1\le r\le t$},\\
s_l & \text{otherwise.}
\end{cases}$$
The swapping map $\Psi_k$ defined in this way preserves the sum $\sum s_l=\sum s_l'=m_k$, hence it is an involution on $\A_n^k({\bf m})$, and $$(\des,\dez)~\a = (\dez,\des)~\mathfrak{b}$$ follows from equations \eqref{des-change} and \eqref{dez-change} immediately. The proof is now completed.
\end{proof}

\begin{example}[An example of $\Psi_1$]
For $k=1$, $\P_n^1=\S_n$. Let 
$$\pi={\bf 1~2}~5~3~9~{\bf6}~4~{\bf 8}~16~11~7~{\bf 12~13}~10~{\bf15}~14~{\bf 17}\in\S_{17}$$ and $\a:=\pf_1^{-1}(\pi)$, then we see
$$\s(\a)=(2,0,0,1,1,0,0,2,1,1)\quad \text{and} \quad \Der(\a)=3~1~5~2~9~7~4~6~8. $$
The type I (resp. type II) slots are $S_1, S_3, S_6$ (resp. $S_2, S_4, S_9$). So we have
$$\s(\Psi(\a))=(2,0,0,1,1,0,1,2,1,0)\quad \text{and} \quad\Der(\Psi(\a))=3~1~5~2~9~7~4~6~8, $$
which gives us $\df_1(\Psi(\a))=\neg 1~\neg 1~3~1~5~\neg 1~2~\neg 1~9~7~\neg 1~4~\neg 1~\neg 1~6~\neg 1~8$. One can verify that indeed
$$(\des, \dez, \Der)~\a = (7, 8, 315297468) = (\dez, \des, \Der)~\Psi(\a).$$
\end{example}

\begin{example}[An example of $\Psi_2$]
For $k=2$, let $\b=(\sigma,\phi)$ with $\sigma=1~2~9~3~5~6~4~8~7$ and $\phi(1)=\phi(8)=\neg 1,\phi(2)=\phi(5)=\phi(6)=\neg 2$, then we see
$$\s(\b)=(0,1,0,2,0,0,0)\quad\text{and}\quad \Der_2(\b)=\neg 1~4~1~2~\neg 1~3.$$
The type I (resp. type II) slots are $S_0,S_2$ (resp. $S_1,S_6$). So we have
$$\s(\Psi(\b))=(1,0,0,2,0,0,0)\quad\text{and}\quad \Der_2(\Psi(\b))=\neg 1~4~1~2~\neg 1~3,$$
which gives us $\df_2(\Psi(\b))=\neg 2~\neg 1~4~1~\neg 2~\neg 2~2~\neg 1~3$. One checks to see
$$(\des, \dez, \Der_2)~\b = (3, 4, \neg 1~4~1~2~\neg 1~3) = (\dez, \des, \Der_2)~\Psi(\b).$$
\end{example}


\section{Bivariate joint generating function for \texorpdfstring{$\des$ and $\dez$}{des and dez}} 
\label{sec:4}
In section \ref{sec:desdez:sym}, we have established that the distribution of the  two Eulerian statistics $\des$ and $\dez$ are symmetric over the permutation group. This section is devoted to the derivation of the bivariate joint generating function for those two statistics. Note that we will abuse the notation to write $\dez(\pi):=\dez(\pf_1^{-1}(\pi))$, for any permutation $\pi$. For instance, $\dez(41352)=\des(41\neg 152)=3$.

A descent (position) $i$ of a permutation $\pi$ is called a {\em crossing descent}, if $\pi_i\ge i+1\ge \pi_{i+1}$. Denote by $\xdes(\pi)$ the number of crossing descents of $\pi$. When restricted to the set of derangements, $\xdes$ is exactly the statistic $t_1$ that we introduce in section~\ref{sec:desdez:sym}. Let
\begin{align}
	\label{def:F}F(t,s; u)&:=1+\sum_{n\geq 2} \Bigl( \sum_{\pi\in \D_n} 
	 t^{\des(\pi)}  s^{\xdes(\pi)} \Bigr) u^n,\\
	\label{def:G}G(x,y; u)&:=1+\sum_{n\geq 1} \Bigl( \sum_{\pi\in{\S}_n} 
	 x^{\des(\pi)} y^{\dez(\pi)} \Bigr) u^n,
\end{align} 
be the generating functions of $(\des, \xdes)$ over the derangement set and $(\des, \dez)$ over the permutation set respectively. The initial values of $F(t,s;u)$ and $G(x,y;u)$ are given below:
\begin{align*}
	F(t,s;u)&= 1+ stu^2 + 2s t u^3 + ( st^3 +  2s^2t^2 + 2st^2 + 4st) u^4 + \cdots , \\
	G(x,y;u)&= 1+ u +    (xy+1)u^2   + (x^2y + xy^2 + 3xy + 1)u^3   \\
	& \qquad + (x^3y^3 + 7x^2y^2 + 4x^2y + 4xy^2 + 7xy + 1)u^4 + \cdots.
\end{align*}
\begin{theorem}\label{th:rational}
	For each nonnegative integer $m$ the coefficient of $t^m$ in $F(t,s;u)$ defined by \eqref{def:F} is a {\it rational fraction} in $s$ and $u$.
\end{theorem}

Theorem \ref{th:rational} is a consequence of Theorem \ref{gfF:frac} 
in view of \eqref{eq:no:deno}.

\begin{theorem}\label{th:gf:desdez}
	We have
	\begin{equation}\label{eq:gf:desdez}
G(x,y;u)= \frac{1}{1-u} \times F\Bigl(
		\frac{xy}{1-u+xyu}, 
		\frac{(1-u+xu) (1-u+yu)}{1-u+xyu};
		u(1+\frac{xyu}{1-u})\Bigr).
	\end{equation}
\end{theorem}
\begin{remark} Depending on $F(t,s;u)$, Theorem \ref{th:gf:desdez} is not explicit. However, we can still see that $G(x,y;u)$ is a symmetric function in $x$ and $y$ from \eqref{eq:gf:desdez}. 
\end{remark}

\begin{proof}
By the definition of $G(x,y;u)$ in~\eqref{def:G}, we have 
\begin{equation}\label{def2:G}
G(x,y; u)=1+\sum_{n\geq 1} \Bigl( \sum_{\a\in\A^1_n} 
	 x^{\des(\a)} y^{\dez(\a)} \Bigr) u^n=\frac{1}{1-u}+\sum_{n\geq 2} \Bigl( \sum_{\a\in\tilde{\A}^1_n} 
	 x^{\des(\a)} y^{\dez(\a)} \Bigr) u^n,
\end{equation}
where $\tilde{\A}^1_n:=\{\a\in\A^1_n: \Der(\a)\neq\emptyset\}$. 
Recall from the last section that any $1$-arrangement $\a\in\tilde{\A}^1_n$ with weak derangement part $\pi=\Der(\a)\in\D_m$ (for some $m\geq2$) has the decomposition 
\begin{equation}\label{decom:1}
\df_1(\a)=B_0\pi_1B_1\pi_2\cdots B_{m-1}\pi_m B_m,
\end{equation}
where each $B_i$ (possibly empty) is  a block with consecutive copies of $\neg1$. We also introduce four types of blocks for $\a$, which are essentially four types of slots of the underlying derangement $\pi$. Note that the first slot of $\pi$ must be of type IV and introduce the type generating function
$$
F'(x,a,b,y; u):=1+\sum_{n\geq 2} \Bigl( \sum_{\pi\in \D_n} 
	 x^{t_1(\pi)} a^{t_2(\pi)} b^{t_3(\pi)} y^{t_4(\pi)-1} \Bigr) u^n.
$$

We aim to connect $F'$ with both $F$ and $G$, so as to establish~\eqref{eq:gf:desdez}.
On the one hand, point~1) of Lemma~\ref{block type} and the discussion preceding it give us
$$t_1(\pi)=t_2(\pi) \quad\text{and}\quad t_1(\pi)+t_2(\pi)+t_3(\pi)+t_4(\pi)-1=m$$
for any $\pi\in\D_m$. Since $\xdes(\pi)=t_1(\pi)$ and $\des(\pi)=t_1(\pi)+t_4(\pi)-1$, we see
\begin{equation}\label{eq:F-F'}
F'(x,a,b,y; u)-1 =\sum_{n\geq 2}  (bu)^n \sum_{\pi\in \D_n} 
	\Bigl(\frac{xa}{b^2}\Bigr)^{t_1(\pi)} \Bigl(\frac{y}{b}\Bigr)^{t_4(\pi)-1} 
	=F\Bigl(\frac{y}{b},\frac{xa}{by};bu\Bigr)-1.
\end{equation}

 On the other hand, invoking the interpretation~\eqref{def2:G} of $G(x,y;u)$, the decomposition~\eqref{decom:1} and relationships  \eqref{des-change} and \eqref{dez-change} give rise to the appropriate substitutions for variables $x,a,b$ and $y$ in $F'$ to arrive at
\begin{equation}\label{eq:G-F'}
G(x,y;u)=\frac{1}{1-u}\times F'\Bigl(xy(1+\frac{xu}{1-u}),1+\frac{yu}{1-u},1+\frac{xyu}{1-u},\frac{xy}{1-u}; u\Bigr),
\end{equation}
where the factor $1/(1-u)$ accounts for the contribution from inserting the block $B_0$. 
Now combining \eqref{eq:F-F'} and \eqref{eq:G-F'} completes the proof.
\end{proof}

It remains to evaluate the generating function $F(t,s;u)$. As it turns out, the following trivariant generalization of $F(t,s;u)$ is more appropriate for calculation:
\begin{equation}
F(t,s,r; u):=\sum_{n\geq 0} \Bigl( \sum_{\pi\in \S_n} 
	 t^{\des(\pi)}  s^{\xdes(\pi)} r^{\fix(\pi)}\Bigr) \frac {u^n}{(1-t)^{n+1}}.
\end{equation}
The reduction to $F(t,s;u)$ is seen to be
\begin{equation*}
F(t,s;u)=(1-s)F(s,ts^{-1},0;(1-s)u).
\end{equation*}
To investigate $F(t,s,r; u)$, we introduce a {\it linear operator} $\rho$ on formal power series in 
$$K[[X,Y,Z]]:=K[[X_0,X_1,X_2,\ldots, Y_0, Y_1, Y_2, \ldots, Z_0, Z_1, Z_2,\ldots]],$$ 
where 
$K=\mathbb{Z}[[s,r,u]]$ and
$X_0,X_1,X_2,\ldots, Y_0, Y_1, Y_2, \ldots, Z_0, Z_1, Z_2,\ldots$ are commuting variables.

\begin{Def}[The operator $\rho$]
For each monomial~$M\in K[[X,Y,Z]]$, the index $i$ is said to be {\em effective} in $M$ if
\begin{itemize}
	\item[i)] $M$ contains $Y_i$ or $Z_i$, and
	\item[ii)] $M$ contains certain $X_k$ with $k>i$ such that
	\item[iii)] $M$ contains neither $Y_j$ nor $Z_j$, for each $i<j<k$.
\end{itemize}
For example, both $0$ and $1$ are effective in $X_1X_2Y_0Y_3Z_0Z_1$, while only $1$ is effective in $X_3X_4Y_1Z_0^2$. Let $\eff(M)$ denote the number of effective indices in $M$. Now we can define the operator $\rho: K[[X,Y,Z]]\rightarrow K$ by setting
\begin{equation}
\rho(M)=s^{\eff(M)},
\end{equation}
and extending linearly to all formal power series in $K[[X,Y,Z]]$. For the previous examples, we have $\rho(X_1X_2Y_0Y_3Z_0Z_1)=s^2$ and $\rho(X_3X_4Y_1Z_0^2)=s$.
\end{Def}

The following theorem can be viewed as the central result of this section.
\begin{theorem}\label{gfF:frac} Let $K=Z[[u]]$. We have
	\begin{equation}\label{eq:rhofrac}
		F(t,s,r ;u) = \sum_{m\geq 0} t^m \rho(S_m(u)),
	\end{equation}
	where
	\begin{equation}\label{def:Sm}
		S_m(u) = 
		\frac{\displaystyle \frac{\prod_{1\leq j \leq m}(1-uX_j)}{\prod_{0\le j\le m}(1-ruZ_j)}}{\displaystyle 1-\sum_{1\leq l \leq m} 
		\frac{u X_l \prod_{1\leq j \leq l-1}(1-u X_j)}{\prod_{0\leq j \leq l-1}(1-uY_j)}}.
	\end{equation}
\end{theorem}

Our strategy to prove Theorem~\ref{gfF:frac} is as follows. First off, we utilize an updated version of the Gessel--Reutenauer standardization \cite{ge, fh, fh3}, denoted as $\Phi_{\mathrm{GR}}$, to map each word $w$ from $[0,m]^n$ onto a pair $(\sigma,c)$, where $\sigma\in\S_n$ and $c=c_1c_2\cdots c_n$ is a word whose letters are nonnegative integers satisfying: $m-\des(\sigma)\ge c_1\ge c_2\ge\cdots\ge c_n\ge 0$. This bijection entitles us to rewrite $F(t,s,r ;u)$ as a weighted (each $w$ weighted by $\psi(w)$, see Definition~\ref{weight:psi}) generating function over all words $w$ in $[0,m]^n$, after we make a key combinatorial observation (see Lemma~\ref{lem:xdes-eff}) to connect the statistic $\xdes$ on a permutation to the statistic $\eff$ on the weight of the corresponding word. Secondly, Theorem~1.3 in \cite{fh} enables us to evaluate this weighted generating function, via the help from the operator $\rho$, to be the right-hand side of \eqref{eq:rhofrac}. 

To begin the first step, we make some definitions and recall the Gessel--Reutenauer bijection. For $n,m\ge 0$, consider the set $\W_n(m):=[0,m]^n$ of all words of length $n$ and alphabet being $[0,m]:=\{0,1,\ldots,m\}$. Denote the subset of {\em {\bf\em n}on-{\bf\em i}ncreasing {\bf\em w}ords} as
$$\NIW_n(m):=\{c=c_1c_2\cdots c_n\in\W_n(m): c_1\ge c_2\ge\cdots\ge c_n\}.$$
We use the lexicographic order ``$>$'' to compare words in $\W_n(r)$. This total order extends to words with different length (but same alphabet) naturally. Namely, let $u\in\W_n(m)$ and $v\in\W_l(m)$ be two nonempty primitive words (none of them can be expressed as $w^b$ for some word $w$ and integer $b\ge 2$), we write $u\succeq v$, if and only if $u^b\ge v^b$ when $b$ is large enough. Here the multiplication is understood to be the concatenation of words.

Let $w=x_1x_2\cdots x_n$ be an arbitrary word over $\Z$ and set $x_{n+1}=+\infty$. For each $1\le i\le n$, we say that $i$ is a {\em decrease} (position) of $w$ if $$x_i= x_{i+1}=\cdots= x_j> x_{j+1},\; \text{for some $i\le j\le n$}.$$ So descent is the case of $i=j$. If on the contrary we have $$x_i= x_{i+1}=\cdots= x_j< x_{j+1},\; \text{for some $i\le j\le n$},$$ then we say that $i$ is an {\em increase} (position) of $w$, and an {\em ascent} (position) of $w$ if $i=j$. By our convention $x_{n+1}=+\infty$, thus $n$ is always an ascent. Furthermore, a position $i$ ($1\le i\le n$) is said to be a {\em record} if $$x_j\le x_i,\;\text{for all}\; 1\le j\le i-1.$$ When the index $i$ is a decrease (resp.~increase, record) of $w$, the corresponding letter $x_i$ is said to be a {\em decrease} (resp.~{\em increase, record}) {\em value} of $w$. The set of all decreases (resp.~increases, ascents, records) is denoted by $\DEC(w)$ (resp.~$\INC(w)$, $\ASC(w)$, $\REC(w)$). In particular, a descent (resp.~ascent) is always a decrease (resp.~increase), thus $\DES(w)\subseteq \DEC(w)$ (resp.~$\ASC(w)\subseteq \INC(w)$). Now we can define the aforementioned weight $\psi$ as was first introduced in~\cite{fh}.

\begin{Def}\label{weight:psi}
Take six sequences of commuting variables $(X_i), (Y_i), (Z_i), (T_i), (Y_i')$ and $(T_i')~(i=0,1,2,\ldots)$, and for each word $w\in\W_n(m)$ define the weight $\psi(w)$ of $w=x_1x_2\cdots x_n$ to be
\begin{align}
\psi(w):=& \prod_{i\in\DES}X_{x_i}\prod_{i\in\ASC\backslash\REC}Y_{x_i}\prod_{i\in\DEC\backslash\DES}Z_{x_i}\\
&\times\prod_{i\in(\INC\backslash\ASC)\backslash\REC}T_{x_i}\prod_{i\in\ASC\cap\REC}Y_{x_i}'\prod_{i\in(\INC\backslash\ASC)\cap\REC}T_{x_i}',\nonumber
\end{align}
where the argument ``$(w)$'' has been suppressed for typographic reasons. For example, if $w=1\;2\;\U{8}\;0\;\U{8}\;2\;10\; \U{13}\;4\;8\; \U{13}\;\U{11}\;\U{11}\;2\;5\;5\; \U{11}\;\U{6}\; \U{3}\;0$ with decrease values underlined, then 
\begin{equation}\label{exm:psi}
\psi(w)=Y'_1Y'_2X_8Y_0X_8Y_2Y'_{10}X_{13}Y_4Y_8X_{13}Z_{11}X_{11}Y_2T_5Y_5X_{11}X_6X_3Y_0.
\end{equation}
\end{Def}

For the sake of convenience, we review the Lyndon factorization of words.
\begin{Def}[Lyndon factorization]
A word $l=x_1x_2\cdots x_n\in\W_n(m)$ is said to be a {\em Lyndon word}~\cite{fh3,lo}, if either $n=1$, or if $n\ge 2$ and $x_1x_2\cdots x_n>x_ix_{i+1}\cdots x_nx_1\cdots x_{i-1}$ holds for every $i$ such that $2\le i\le n$. As shown for instance in \cite[Theorem~5.1.5]{lo}, each nonempty word $w$ composed of nonnegative integers, can be written uniquely as a product $w=l_1l_2\cdots l_k$, where each $l_i$ is a Lyndon word and $l_1\preceq l_2\preceq \cdots\preceq l_k$. This word factorization is called  {\em Lyndon factorization}. For instance, we have the Lyndon factorization 
$$
w=1\,2\,1\,0\,0\,2\,2\,4\,5\,3\,1\,0\,2\,1\,2\,5=1\,|\,2\,1\,0\,0\,|\,2\,|\,2\,|\,4\,|\,5\,3\,1\,0\,2\,1\,2\,|\,5,
$$
where factors are separated by vertical bars. 
\end{Def}
Finally, we recall the construction of the inverse $\Phi_{\mathrm{GR}}^{-1}:(\sigma,c)\mapsto w$ by means of one example. A description of this correspondence with more details can be found in Foata and the second author's previous paper~\cite{fh3}.

{
\small
\setlength{\tabcolsep}{3pt}
\begin{center}
%
\begin{tabular}{c c c  c c c c c c c c c c c c c c c c c c c c c}
Id &=& 1&2&3&4&5&6&7&8&9&10&11&12&13&14&15&16&17&18&19&20 \\
$\rightarrow\sigma$ &=& {\U 3}&\U{13}&{\U 5}&\U{10}&\U{16}&{\B 6}&2& \U{15}& \U{20}& \U{14}&4&11&7& \U{19}&8&12& \B{17}& \B{18}&1&9 \\
$z $ & =& 8&8&7&7&7&6&5&5&5&4&3&3&2&2&1&1&1&1&0&0 \\
$\rightarrow c $ & =& 5&5&4&4&4&4&3&3&3&2&2&2&2&1&1&1&1&0&0&0 \\
$\bar c$ &=& 13&13&11&11&11&10&8&8&8&6&5&5&4&3&2&2&2&1&0&0 \\
$\sigma $&=& (\B{18})& (\B{17}) &(9&20)&(8&15)& (\B{6})&(2&13&7)&(1&3&5&16&12&11&4&10&14&19) \\
$\check\sigma$&=& \B{18}& \B{17}&9&20&8&15& \B{6}&2&13&7&1&3&5&16&12&11&4&10&14&19 \\
$\mapsto w$& =& 1\;|&2\;|& \U{8}&0\;|& \U{8}&2\;|&10\;|& \U{13}&4&8\;|& \U{13}& \U{11}& \U{11}&2&5&5& \U{11}& \U{6}& \U{3}&0\\
\end{tabular}
\end{center}
}

In above example $n=20$. The second row contains the values $\sigma(i)$ ($i=1,2,\ldots,n$) of the {\em starting} permutation $\sigma$. The fixed points in $\sigma$ are written in boldface, while the {\em excedances} $\sigma(i)>i$  are underlined. The third row is the vector $z=z_1z_2\cdots z_n$ defined as
\begin{equation}\label{def:z-seq}
z_i:=|\{j:i\le j\le n-1,\;\sigma(j)>\sigma(j+1)\}|, \; \text{for}\; 1\le i\le n,
\end{equation}
so that $z_1=\des(\sigma)$. The fourth row is the {\em starting} nonincreasing word $c=c_1c_2\cdots c_n$. The fifth row $\bar{c}=\bar c_1 \bar c_2 \cdots\bar c_n$ is the word defined by
\begin{equation*}
\bar c_i:=z_i+c_i,\; \text{for}\; 1\le i\le n.
\end{equation*}
The sixth row is again the permutation $\sigma$ but now in its {\em cycle notation}, with the minima leading each cycle and cycles listed with their first letters decreasing. When removing the parentheses in the sixth row we arrive at the seventh row denoted as $\check\sigma=\check\sigma(1)\check\sigma(2)\cdots\check\sigma(n)$. The bottom row is the word $w=x_1x_2\cdots x_n$ corresponding to the pair $(\sigma,c)$ defined by 
\begin{equation}\label{eq:w-barc}
x_i:=\bar c_{\check\sigma(i)},\;\text{for}\; 1\le i\le n.
\end{equation}
The underlined letters in $w$ are decrease values of $w$. Finally, the vertical bars inserted into $w$ indicate its Lyndon factorization.

It is known (cf.~\cite{fh,fh3}) that all the above steps are reversible and $\Phi_{\mathrm{GR}}:w\mapsto(\sigma,c)$ is indeed a bijection, essentially due to Gessel and Reutenauer~\cite{ge}, from $\W_n(m)$ onto the set of pairs $(\sigma,c)$ such that $\sigma\in\S_n$, $\des(\sigma)\le m$ and $c\in\NIW_n(m-\des(\sigma))$. The following observation was made in~\cite{fh}.

\goodbreak

\begin{Observation}\label{obser:psi}
Suppose $\Phi_{\mathrm{GR}}(w)=(\sigma,c)$, then we have
\begin{itemize}
	\item[(i)] $i\in\DEC(w)$ if and only if $\check\sigma(i)<\check\sigma(i+1)$, 
	i.e., $\check\sigma(i)$ is an excedance of $\sigma$.
	\item[(ii)] $i\in\INC(w)\cap\REC(w)$ if and only if $\check\sigma(i)\in\Fix(\sigma)$.
\end{itemize}
\end{Observation}

Guided by Observation~\ref{obser:psi}, we let $\gamma$ be the homomorphism defined by the following substitutions of variables:
	\begin{equation}\label{subs}
	\gamma:=\{X_j\leftarrow u X_j, \
	Z_j\leftarrow u X_j, \
	Y_j\leftarrow u Y_j, \
	T_j\leftarrow u Y_j, \
	Y'_j\leftarrow ru Z_j, \
	T'_j\leftarrow ru Z_j\}.
	\end{equation}
The following feature of $\Phi_{\mathrm{GR}}$ regarding crossing descents is key to our calculation. 
\begin{lemma}\label{lem:xdes-eff}
Suppose $\Phi_{\mathrm{GR}}(w)=(\sigma,c)$, then $i$ is a crossing descent of $\sigma$, if and only if the index $\bar c_{i+1}$ is effective in $\gamma\psi(w)$.
\end{lemma}

Take $w$ as in the running example of $\Phi_{\mathrm{GR}}^{-1}$  above. By the calculation in~\eqref{exm:psi}, we have 
$$
\gamma\psi(w)=u^{20}X_3X_6X_8^2X_{11}^3X_{13}^2Y_0^2Y_2^2Y_4Y_5^2Y_8Z_1Z_2Z_{10},
$$
and so the effective indices in $\gamma\psi(w)$ are $2$, $5$ and $10$, which correspond respectively to the crossing descents $14$, $10$ and $5$ of the permutation $\sigma$.

\begin{proof}[{\bf Proof of Lemma~\ref{lem:xdes-eff}}]
First we show the  ``only if'' part. Suppose $i$ is a crossing descent of $\sigma$, i.e., $\sigma(i)\ge i+1\ge \sigma(i+1)$. Note that $i$ is a descent of $\sigma$ so $z_i=z_{i+1}+1$ hence $\bar c_i>\bar c_{i+1}$. In view of  Observation~\ref{obser:psi} (i),  we have 
\begin{itemize}
\item $i+1\ge \sigma(i+1)$ means that $i+1=\check\sigma(\check\sigma^{-1}(i+1))$ is not an excedance of $\sigma$, which implies that either $Y_{\bar c_{i+1}}$ or $Z_{\bar c_{i+1}}$ appears in $\gamma\psi(w)$, and
\item $\sigma(i)>i$ means that $i=\check\sigma(\check\sigma^{-1}(i))$ is an excedance of $\sigma$, which indicates that $X_{\bar c_i}$  appears in $\gamma\psi(w)$. 
\end{itemize}
By definition this means that $\bar c_{i+1}$ is effective in $\gamma\psi(w)$.

It remains to show the ``if'' part. Conversely, suppose that certain  $a$  is effective in $\gamma\psi(w)$. Then, we can find indices $i\geq1$, $j\geq1$ and $k\geq0$ such that  $b=\bar c_i>a=\bar c_{i+1}=\bar c_{i+2}=\cdots=\bar c_{i+j}>\bar c_{i+j+1}$ and $\bar c_{i-k-1}>\bar c_{i-k}=\bar c_{i-k+1}=\cdots=\bar c_i=b$.  We aim to show that $i$ is a crossing descent of $\sigma$. Since $a$ is effective in $\gamma\psi(w)$, at least one of $Y_a$ and $Z_a$ appears in $\gamma\psi(w)$, which implies that one of $i+\ell$, $1\leq \ell\leq j$, must be a non-excedance of $\sigma$. This forces $i+1$ to be a non-excedance, as $\sigma(i+1)<\sigma(i+2)<\cdots<\sigma(i+\ell)$. On the other hand, we must have $X_b$ appear in $\gamma\psi(w)$, which implies one of $i-\ell'$, $0\leq \ell'\leq k$, must be an excedance of $\sigma$.
This forces $i$ to be an excedance of $\sigma$, as $\sigma(i-\ell')<\sigma(i-\ell'+1)<\cdots<\sigma(i)$. 
In conclusion, $i$ is a crossing descent of $\sigma$, as desired. 
\end{proof}

We will also make use of the following version of the so-called ``Decrease Value Theorem''.
\begin{theorem}[Theorem 1.3 in \cite{fh}]
We have:
\begin{equation}\label{eq:decthm}
\sum_{n\ge 0}\sum_{w\in\W_n(m)} \psi(w)
=\dfrac{\dfrac{\prod\limits_{1\le j\le m}\dfrac{1-Z_j}{1-Z_j+X_j}}{\prod\limits_{0\le j\le m}\dfrac{1-T'_j}{1-T'_j+Y'_j}}}
{1-\sum\limits_{1\le l\le m}\dfrac{\prod\limits_{1\le j\le l-1}\dfrac{1-Z_j}{1-Z_j+X_j}}{\prod\limits_{0\le j\le l-1}\dfrac{1-T_j}{1-T_j+Y_j}}\dfrac{X_l}{1-Z_l+X_l}}.
\end{equation}
\end{theorem}
We are in a position to prove Theorem~\ref{gfF:frac}.
\begin{proof}[{\bf Proof of Theorem~\ref{gfF:frac}}]
	For each word $w=x_1x_2\cdots x_n \in \W_n(m)$,
	we have:
	\begin{equation}\label{eq:gammapsiw}
		\gamma\psi(w)= 
			{\displaystyle	u^n \prod_{i\in \DEC} {X_{x_i}} \prod_{i\in \INC\setminus\REC} {Y_{x_i}} \prod_{i\in \INC \cap \REC} {Z_{x_i}} }.
	\end{equation}
	Applying $\gamma$ to \eqref{eq:decthm} we get:
	\begin{equation}\label{eq:XYw}
		\sum_{n\ge 0}\sum_{w\in\W_n(m)} u^n \prod_{i\in \DEC} {X_{x_i}} \prod_{i\in \INC\setminus\REC} {Y_{x_i}} \prod_{i\in \INC \cap \REC} {Z_{x_i}}  = 
		S_m(u).
	\end{equation}
	By $\Psi_{\mathrm{GR}}: w\mapsto (\sigma, c)$ and Observation~\ref{obser:psi}, the left-hand side of \eqref{eq:XYw} is equal to
	\begin{equation}\label{eq:desxdes:perm}
		\sum_{n\ge 0} u^n \sum_{\substack{\sigma \in \S_n \\ \des(\sigma)\leq m}}   
		\sum_{c\in \NIW_n(m-\des(\sigma))} r^{\fix(\sigma)} W_{(\sigma,c)},
	\end{equation}
	where 
	$$
	W_{(\sigma,c)} = \prod_{j<\sigma(j)} X_{c_j+z_j} \prod_{j>\sigma(j)} Y_{c_j+z_j} \prod_{j=\sigma(j)} Z_{c_j+z_j}.
	$$
	With
	$$
	W(\sigma;t) := \sum_{k\geq 0} t^k \sum_{c \in \NIW_n(k)} W_{(\sigma,c)},
	$$
	the graded form of \eqref{eq:desxdes:perm} reads
	\begin{equation}\label{eq:desxdes:graded}
		\sum_{n\geq 0} u^n \sum_{\sigma \in \S_n}   t^{\des(\sigma)} r^{\fix(\sigma)}W(\sigma;t)
		= \sum_{m\geq 0} t^m S_m(u).
	\end{equation}
	Now Lemma~\ref{lem:xdes-eff} says that $i$ is a crossing descent of $\sigma$, if and only if $\bar c_{i+1}$ is effective in $W_{(\sigma,c)}$, therefore we have
	\begin{equation}\label{eq:rhoY}	
	\rho(W_{(\sigma, c)})= s^{\eff(W_{(\sigma, c)})} = s^{\xdes(\sigma)}.
	\end{equation}	
    Consequently,
	$$
	\rho(W(\sigma;t))= s^{\xdes(\sigma)}  \sum_{k\geq 0} t^k\times |\NIW_n(k)|
	=\frac{s^{\xdes(\sigma)} }{(1-t)^{n+1}}.
	$$
	Applying $\rho$ to both sides of \eqref{eq:desxdes:graded} yields \eqref{eq:rhofrac}.
\end{proof}

Although Theorem \ref{gfF:frac} is somewhat complicated, it allows us to derive some 
formulae for special cases, with the help of a computer algebra system.
A trick to evaluate the fraction at the right-hand side of \eqref{eq:rhofrac} by the operator $\rho$ is that we can replace $x^k$, for $x=X_j$,  $Y_j$ or $Z_j$, by $x$ for $k\geq 1$. That is
\begin{equation}\label{rho:trick}
\rho\Bigl(\frac{1}{a+bx}\Bigr) = \rho\Bigl(\frac{1}{a}(1-\frac{bx}{a+b})\Bigr).
\end{equation}
Using this trick, we can derive the formulae for $m=1,2$ as follows. 

{\bf (I)} Special case $m=1$. The term $\rho(S_1(u))$ is equal to 
\begin{align*}
	&\phantom{==}\rho\left( \frac{1-uX_1}{(1-ruZ_0)(1-ruZ_1)} \middle/ {(1- \frac{u X_1}{ 1-uY_0})} \right)  \\
	&=
	\rho\left( \frac{(1-uX_1)(1-uY_0)}{(1-ruZ_0)(1-ruZ_1) (1- u X_1-uY_0) }\right)\\
	&=
	\frac{1}{(1-ru)^2}\left(1 + \rho( \frac{u^2X_1Y_0}{1- u X_1-uY_0})\right)= \frac{1}{(1-ru)^2}
	\Bigl(1 + \frac{su^2}{ 1- 2u}\Bigr).
\end{align*}
We now give a combinatorial proof of the formula above. When we extract the coefficient of $t^1$ from \begin{equation*}
F(t,s,r; u)=\sum_{n\geq 0} \Bigl( \sum_{\sigma\in \S_n} 
	 t^{\des(\sigma)}  s^{\xdes(\sigma)} r^{\fix(\sigma)}\Bigr) \frac {u^n}{(1-t)^{n+1}},
\end{equation*} 
there are two cases to be considered:
\begin{enumerate}
	\item The factor $1/{(1-t)^{n+1}}$ contributes $(n+1)t^1$ while the permutation satisfies $\des(\sigma)=0$. Then the permutation must be the identity and contributes $r^nu^n$. So we sum up over all $n\ge 0$ to get the term $1/{(1-ru)^2}$.
	\item The factor $1/{(1-t)^{n+1}}$ contributes $t^0=1$ while the permutation satisfies $\des(\sigma)=1$. Now note that any permutation $\sigma\in\S_n$ that has exactly one descent (the so-called ``Grassmannian permutation''), can be uniquely decomposed as $$\sigma=12\cdots i~\sigma_{i+1}\sigma_{i+2}\cdots\sigma_{i+j}~(i+j+1)(i+j+2)\cdots n,$$
	where $0\le i \le n-2$, $2\le j\le n-i$, and $\sigma_{i+1}\cdots\sigma_{i+j}$ contains no fixed points and exactly one descent at $i+j=\sigma_{k}>\sigma_{k+1}=i+1$ for certain $k$ ($i+1\le k<i+j$), hence this descent is a crossing descent. All the other letters $i+2,i+3,\ldots,i+j-1$ can appear either to the left of $i+j$, or to the right of $i+1$, making $2^{j-2}$ choices in total. So we see the contributions from all such $\sigma$'s are
	$$\frac{s}{(1-ru)^2}\cdot \Bigl(\sum_{j\ge 2}2^{j-2}u^j\Bigr)=\frac{su^2}{(1-ru)^2(1-2u)},$$
	which is precisely the remaining term in the formula above.
\end{enumerate}
 
{\bf (II)} Special case $m=2$. With the help of a computer algebra program, we obtain the term $\rho(S_2(u))$, which is equal to

 \begin{align*}
 	&\frac{1}{(1-ru)^3}
 	\Bigl(	1 +
 	 \frac{{\left((3-r)  u^{2} - 7  u + 3\right)} u^{2}\cdot s}{{\left(u^{2} - 3  u + 1\right)} {\left(1- 2  u \right)}}
 	+
 	\frac{{\left(-3  r u + 2  u^{2} + r - 4  u + 2\right)}   u^{4}\cdot s^2}{{\left(u^{2} - 3  u + 1\right)} {\left(1-3  u \right)} {\left(1-2  u \right)}}
 	\Bigr).
 \end{align*}

Let $\S_n^{(m)} = \{ \pi \in \S_n : \des(\pi)=m\}$ and 
\begin{equation}\label{def:Pj}
	P_m(u) =  \sum_{n\geq 0}\Bigl(	\sum_{\pi \in \S_n^{(m)}} s^{\xdes(\pi)} r^{\fix(\pi)}\Bigr) u^n
\end{equation}
be the generating function for the $\xdes$ and $\fix$ statistics over the 
set of all permutations with exactly $m$ descents.
Theorem \ref{gfF:frac} implies that
\begin{equation}\label{eq:no:deno}
\sum_{n\geq 0} \Bigl( \sum_{\pi\in \S_n} 
	 t^{\des(\pi)}  s^{\xdes(\pi)}  r^{\fix(\pi)}\Bigr)  {u^n}
	 = (1-t)\sum_{m\geq 0} t^m \rho(S_m((1-t)u)).
\end{equation}
Comparing the coefficients of $t^0, t^1, t^2$ in \eqref{eq:no:deno} with
the explicit values of $\rho(S_1(u))$ and $\rho(S_2(u))$ previously obtained, we derive
\begin{align*}
	P_0(u)
		&=  \frac{1}{1- r u }; \\
		P_1(u)
		&= 
\frac{u^{2} s}{{\left(1-r u \right)}^{2} {\left(1-2 \, u \right)}};
		 \\
		P_2(u)
		&=  
		\frac{u^3}{(1-ru)^3(u^2-3u+1)(1-2u)}  \times \\
		& \qquad \left(
	\frac{ {\left(3 \, r u - r - 2 \, u\right)} {\left(u - 1\right)}  }{ {1-2 u }} s 
	+ \frac{{\left(-3 \, r u + 2 \, u^{2} + r - 4 \, u + 2\right)} u }{1-3u} s^2 \right).\nonumber
\end{align*}

To end this section, we connect our results with two kinds of Genocchi numbers, and the new statistic $\xdes$ with earlier work of Ehrenborg and Steingr\'imsson~\cite{es}.

$\bullet$ For $m=1,2,\ldots$, the coefficient of $[s^m u^{2m}]$ in $\rho(S_m(u))|_{r=0}$ are
$$
1,2,8,56,608, \ldots 
$$
which are the Genocchi numbers, second kind, or Genocchi medians.
In fact, by Theorem~\ref{gfF:frac} this coefficient is the the coefficient of
$[t^m s^m u^{2m}]$ in $F(t,s,0;u)$, i.e., the coefficient of $[t^m s^m]$ in
\begin{equation*}
	\Bigl( \sum_{\pi\in D_{2m}} 
	 t^{\des(\pi)}  s^{\xdes(\pi)} \Bigr) \frac {1}{(1-t)^{2m+1}},
\end{equation*}
which is equal to the number of derangements $\pi$ in $D_{2m}$ such that $\des(\pi)=\xdes(\pi)=m$. By \cite{Dumont1994Rand} we know that the Genocchi number of second kind 
is the number of derangements $\sigma$ on $\{1,2,\cdots, 2m\}$ such that 
for $\sigma(i)>i$ iff $i$ is odd, which is equivalent to the condition $\des(\sigma)=\xdes(\sigma)=m$.

\medskip

$\bullet$ For $m=1,2,\ldots$, the coefficient of $[s^m u^{2m}]$ in 
$\rho(R_m(u))|_{r=1}$ are
$$
1,3,17,155,2073, \ldots 
$$
which are the Genocchi numbers, first kind. In fact, 
by Theorem~\ref{gfF:frac} this coefficient is the the coefficient of
$[t^m s^m u^{2m}]$ in $F(t,s,1;u)$, i.e., the coefficient of $[t^m s^m]$ in
\begin{equation*}
	\Bigl( \sum_{\pi\in \S_{2m}} 
	 t^{\des(\pi)}  s^{\xdes(\pi)} \Bigr) \frac {1}{(1-t)^{2m+1}},
\end{equation*}
which is equal to the number of permutations $\pi$ in $\S_{2m}$ such that $\des(\pi)=\xdes(\pi)=m$. By \cite{Dumont1974,Dumont1994Rand} we know that the Genocchi number of first kind 
is the number of permutations $\sigma$ on $\{1,2,\cdots, 2m\}$ such that 
for $\sigma(i)>i$ iff $i$ is odd, which is equivalent to the condition $\des(\sigma)=\xdes(\sigma)=m$.
\medskip

$\bullet$ Although the definition of $\xdes$ might seem a little peculiar, it has disguisedly showed up in the literature. In~\cite[Def.~4.1]{es}, Ehrenborg and Steingr\'imsson introduced the notion of ``excedance run'' on $ab$-words (certain equivalence classes on permutations determined by their excedance sets), which is essentially the same as our $\xdes$, defined on permutations. More precisely, for any permutation $\pi$, the number of crossing descents of $\pi$ equals the number of (excedance) runs of the $ab$-word of $\pi$.


\section{Patterns on \texorpdfstring{$k$-arrangements}{k-arrangements}}
\label{sec:5}
In this section, we denote $\max(w)$ and $\min(w)$ the maximal and the minimal letters of a word $w$ over integers, respectively.
\subsection{Pattern avoiding \texorpdfstring{$3$-arrangements}{3-arrangements} in permutation form}
The following enumeration result was suspected in~\cite[Conj.~2]{bs}.
\begin{theorem}
The number of $3$-arrangements of $[n]$ whose permutation form avoids any single pattern of length 3 is $C(n + 2)-2^n$.
\end{theorem}
\begin{proof}
First of all, it was shown by Savage and Wilf~\cite[Thm.~3]{sav} that the number of permutations (or rearrangements) of a given multiset that avoid a pattern of length $3$ is independent of the pattern. The same holds true for the permutation form of $k$-arrangements for any $k\geq1$, since $\P_n^k$ is  a union of
rearrangement classes~\cite[Prop.~3]{bs}. Thus, it is sufficient to prove 
\begin{equation}\label{312:3}
C^{(3)}(x):=1+\sum_{n\geq1}|\P_n^3(312)|x^n=\sum_{n\geq0}(C(n+2)-2^n)x^n=\frac{C(x)-1-x}{x^2}-\frac{1}{1-2x},
\end{equation}
where $C(x)$ is the generating function for Catalan numbers
$$
C(x):=\sum_{n\geq0}C(n)x^n=\frac{1-\sqrt{1-4x}}{2x}. 
$$

On the other hand, Blitvi\'c and Steingr\'imsson~\cite[Prop.~6]{bs} showed that 
\begin{equation}\label{312:2}
C^{(2)}(x):=1+\sum_{n\geq1}|\P_n^2(312)|x^n=\frac{C(x)-1}{x}. 
\end{equation}
We view $\P_n^3(312)$ as the disjoint union of $\P_n^2(312)$ and $\bar{\P}_n^3(312):=\P_n^3(312)\setminus\P_n^2(312)$.  
Any $w=w_1\ldots w_n\in\bar{\P}_n^3(312)$ with $w_j$ being the rightmost letter $\neg2$ can be written  as $w=\alpha\,\neg2\,\beta$ with $\max(\alpha)\leq \min(\beta)$. This decomposition can be fully characterized according to the following two cases:
\begin{itemize}
\item if $\beta$ contains the letter $\neg 1$, then we have $\reduction(\beta)\in\P_{n-j}^2(312)\setminus\P_{n-j}^1(312)$ and $\alpha\in\{\neg1,\neg2\}^{j-1}$;
\item otherwise $\beta$ has purely positive letters, and $\reduction(\beta)\in\P_{n-j}^1(312)$, $\alpha\in\P_{n-j}^3(312)$. 
\end{itemize}
This decomposition is reversible and in terms of generating function gives 
$$
C^{(3)}(x)-C^{(2)}(x)=x(C^{(2)}(x)-C(x))(1-2x)^{-1}+xC^{(3)}(x)C(x).
$$
Combining this with~\eqref{312:2} yields~\eqref{312:3} after simplification using Maple. 
\end{proof}

\subsection{The statistic \texorpdfstring{$\des$ on pattern avoiding $2$-arrangements}{des on pattern avoiding 2-arrangements} in permutation form}
This subsection is devoted to the classification of the $\des$-Wilf equivalences  for patterns of length $3$ for permutation form of  $2$-arrangements. These $\des$-Wilf equivalences were stated as Conjectures~3 and~4 in~\cite[Sec.~3.6]{bs}.

\begin{theorem}
The distribution of $\des$ on $2$-arrangements of $[n-1]$ whose permutation form avoids any single one of the patterns $213,312,231$ or $132$, is given by the triangle sequence A108838 in~\cite{oeis}, which counts, among other things, parallelogram polyominoes of semiperimeter $n+1$ having $k$ corners, and has formula $\frac{2}{n+1}\binom{n+1}{k+2}\binom{n-2}{k}$. 
\end{theorem}

\begin{proof}
It is known (cf.~\cite[Sec.~2.3]{pe}) that the size generating function 
$$
N=N(t,x):=1+\sum_{n\geq1}x^n\sum_{\pi\in\S_n(\sigma)}t^{\des(\pi)},
$$
where $\sigma$ is one of the patterns $213,312,231$ or $132$, satisfies the functional equation 
\begin{equation}\label{eq:nara}
txN^2-(1-x+tx)N+1=0.
\end{equation}

First we consider the pattern $312$. Any $w=w_1\cdots w_n\in\P_n^2(312)\setminus\P_n^1(312)$ with $w_j$ being the rightmost letter $\neg1$  can be written as $w=\alpha\,\neg1\,\beta$ with $\max(\alpha)\leq \min(\beta)>0$ such that $\alpha\,\neg1\in\tilde{\P}_{j}^2(312)$ and $\reduction(\beta)\in\P_{n-j}^1(312)=\S_{n-j}(312)$, where $\tilde{\P}_j^2(312)$ denotes the set of words $w\in\P_j^2(312)$ whose last letter is $\neg1$. Moreover, we have 
$$
\des(w)=\des(\alpha\,\neg1)+\des(\beta)\quad\text{and}\quad\des(\alpha\,\neg1)=\des(\alpha)+\chi(w_{j-1}\neq\neg1),
$$
where $\chi(\mathsf{S})$ equals $1$, if the statement $\mathsf{S}$ is true; and $0$, otherwise. Let us introduce 
$$
F(t,x):=1+\sum_{n\geq1}x^n\sum_{w\in\P_n^2(312)}t^{\des(w)}\quad\text{and} \quad G(t,x):=\sum_{n\geq1}x^n\sum_{w\in\tilde{\P}_n^2(312)}t^{\des(w)}.
$$
 The above decomposition then gives the system of equations
$$
\begin{cases}
\,\,F=N+GN,
\\
\,\,G=tx(F-1-G)+x(1+G).
\end{cases}
$$
Solving this system of equations yields $N=\frac{(1-x+tx)F}{1+txF}$. Substituting this into~\eqref{eq:nara} results in
\begin{equation}\label{eq:312}
t^2x^2F^2-(t^2x^2-2tx^2+x^2-2x+1)F+1=0.
\end{equation}
Comparing with the generating function for sequence A108838 in~\cite{oeis} proves the desired result for pattern $312$. The proof for the pattern $213$ is identical and will be omitted. 

Next we consider the pattern $231$. Let $\bar{\P}_n^2(231)$ denote the set of $w\in\P_n^2(231)$ with $\max(w)\neq\neg1$. Any $w=w_1\cdots w_n\in\bar{\P}_n^2(231)$ with the largest letter $w_j=m>0$ can be written as $\alpha\,m\,\beta$, where $\max(\alpha)\leq \min(\beta)$. We have two cases:
\begin{itemize}
\item if $\max(\alpha)=\neg1$ (or $\alpha$ is empty), i.e., $\alpha$ is a word with all letters being $\neg1$, then $\reduction(\beta)\in\P_{n-j}^2(231)$ (possibly empty);
\item otherwise $\max(\alpha)\ge 1$, then $\alpha\in\bar{\P}_{j-1}^2(231)$ and $\reduction(\beta)\in\P_{n-j}^1(231)$ (possibly empty). 
\end{itemize}
By this decomposition, if we define 
$$
H=H(t,x):=1+\sum_{n\geq1}x^n\sum_{w\in\P_n^2(231)}t^{\des(w)},
$$
then 
$$
H=\frac{1}{1-x}+\frac{x(1+t(H-1))}{1-x}+\biggl(H-\frac{1}{1-x}\biggr)x(1+t(N-1)). 
$$
Thus, we have
$N=\frac{H(tx^2-x^2+2x-1)+1}{tx(Hx-H+1)}$. Substituting this into~\eqref{eq:nara} results in
$$
t^2x^2H^2-(t^2x^2-2tx^2+x^2-2x+1)H+1=0,
$$
which proves the statement for the pattern $231$ after comparing with~\eqref{eq:312}. The proof for the pattern $132$ is the same as for $231$ and thus is omitted. The proof of the theorem is now completed. 
\end{proof}

Finally, we deal with the patterns $321$ and $123$, thus completing the classification of all six patterns of length $3$, in terms of their $\des$-Wilf equivalences on $\P_n^2$. Our proof of the following connection is algebraic. A bijective proof would be interesting. 

\begin{theorem}\label{thm:321}
The distribution of $\des$ on $2$-arrangements of $[n-1]$ whose permutation form avoids the pattern $321$, is given by the triangle sequence A236406 in~\cite{oeis}, which counts $321$-avoiding permutations of $[n]$ with $k$ peaks.
\end{theorem}

In order to prove Theorem~\ref{thm:321}, we need to compute the joint distribution of the number of descents and the position of the leftmost descent on $321$-avoiding permutations. We will apply Krattenthaler's classical bijection~\cite{kr} (see also~\cite{eli}) from Dyck paths to $321$-avoiding permutations. 
   
A {\em Dyck path} of semilength $n$ is a lattice path in $\N^2$ from $(0,0)$ to
$(n,n)$ using the {\em east step} $(1,0)$ and the {\em north step} $(0,1)$, which does
not pass above the line $y=x$. The {\em height of an east step} in a Dyck path is the number of north steps before this east step. For the sake of convenience, we represent a Dyck path  as $d_1d_2\cdots d_n$, where $d_i$ is the height of its $i$-th east step. See Fig.~\ref{ex:varphi} for the Dyck path $012224566$. Denote by $\mathcal{D}_n$ the set of all Dyck paths of semilength $n$. In particular, we denote by $$\ID_n=01\cdots n-1\;\text{and}\; \id_n=12\cdots n$$ the zigzag Dyck path of semilength $n$ and the identity permutation of length $n$, with $\ID_0$ and $\id_0$ being the empty path and the empty permutation, respectively. We will use the description of Krattenthaler's bijection  $\psi:\mathcal{D}_n\rightarrow\S_n(321)$ in~\cite{kl}. Given a Dyck path $D=d_1d_2\cdots d_n\in\mathcal{D}_n$,  define $\psi(D)=\pi=\pi_1\pi_2\cdots\pi_n$, where 
\begin{itemize}
\item $\pi_i=d_i+1$ if $d_i\neq d_{i+1}$ or $i=n$; otherwise
\item  if $i$ is the $j$-th smallest integer in $\{k\in[n-1]: d_k=d_{k+1}\}$, then $\pi_i$ is the $j$-th smallest integer in $[n]\setminus\{d_1+1,d_2+1,\ldots,d_n+1\}$.
\end{itemize}
See Fig.~\ref{ex:varphi} for  a  visualization of this bijection for the Dyck path $012224566$.
\begin{figure}
\begin{center}
\begin{tikzpicture}[scale=.5]
\draw[step=1,color=gray](0,0) grid (9,9);
\draw [very thick](0,0)--(1,0)--(1,1)--(2,1)--(2,2)--(5,2)--(5,4)--(6,4)--(6,5)--(7,5)--(7,6)--(9,6)--(9,9);
\draw(.5,-0.5) node{{$1$}};\draw(1.5,-0.5) node{{$2$}};
\draw(2.5,-0.5) node{{$\bf 4$}};\draw(3.5,-0.5) node{{$\bf 8$}};
\draw(4.5,-0.5) node{{$3$}};
\draw(5.5,-0.5) node{{$5$}};\draw(6.5,-0.5) node{{$6$}};
\draw(7.5,-0.5) node{{$\bf 9$}};
\draw(8.5,-0.5) node{{$7$}};
\draw(8.5,6.5) node{{$\times$}};
\draw(6.5,5.5) node{{$\times$}};
\draw(5.5,4.5) node{{$\times$}};
\draw(4.5,2.5) node{{$\times$}};
\draw(1.5,1.5) node{{$\times$}};
\draw(0.5,0.5) node{{$\times$}};
\draw(7.5,8.5) node{{$\circ$}};
\draw(3.5,7.5) node{{$\circ$}};
\draw(2.5,3.5) node{{$\circ$}};
\end{tikzpicture}
\end{center}
\caption{Krattenthaler's bijection $\psi:\mathcal{D}_n\rightarrow\S_n(321)$.}
\label{ex:varphi}
\end{figure}
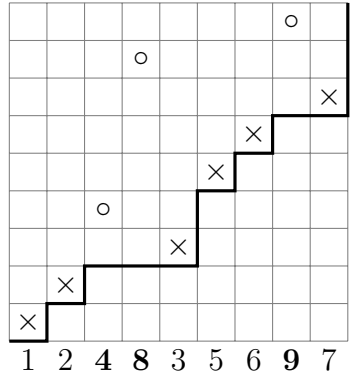

Let us introduce the following three statistics for $D\in\mathcal{D}_n$:
\begin{itemize}
\item $\hill(D)$, the number of hills of $D$, where a {\em hill} of a Dyck path is an east step touching the diagonal $y=x$ and followed immediately by a north step.
\item $\seg(D)$, the number of {\bf seg}ments of $D$, where a {\em segment} is a maximal string of at least two consecutive east steps of the same height;
\item $\lseg(D)=i$,  if the $i$-th east step is the  last  step of the {\bf l}eftmost {\bf seg}ment  of $D$. Otherwise, $D=\ID_n$ has no segments, then we let $\lseg(D)=n+1$. In particular, $\lseg(\ID_0)=1$. 
\end{itemize}
Continuing with our Dyck path in Fig.~\ref{ex:varphi}, we have  $\hill(D)=\seg(D)=2$ and $\lseg(D)=5$. For a permutation $\pi\in\S_n$, let 
\begin{equation}\label{def:ldes}
\ldes(\pi):=\min\{i:\pi_i>\pi_{i+1} \text{ or $i=n$ }\}
\end{equation}
be the the position of the {\em{\bf\em l}eftmost {\bf\em des}cent} of a permutation $\pi$. In particular, $\ldes(\id_0)=0$. The following property is clear from the above description of $\psi$. 

\begin{lemma}\label{lem:kra}
For each $n\ge 0$, the bijection $\psi:\mathcal{D}_n\rightarrow\S_n(321)$ transforms  $(\hill,\seg,\lseg) D$ to $(\fix,\des,\ldes+1)\psi(D)$. 
\end{lemma}

We continue to compute the generating function  
\begin{align*}
C(t,p)=C(t,p;x)&:=p+\sum_{n\geq1}x^n\sum_{D\in\mathcal{D}_n}t^{\seg(D)}p^{\lseg(D)}=p+p^2x+(p+t)p^2x^2+\cdots\\
&=\frac{p}{1-px}+\sum_{n\geq1}x^n\sum_{\substack{\pi\in\S_n(321) \\ \pi\neq\id_n}}t^{\des(\pi)}p^{\ldes(\pi)+1}
\end{align*}
using the {\em classical decomposition} of Dyck paths. 

\begin{lemma}\label{decom:dyck}
The generating function $C(t,p;x)$ satisfies the algebraic functional equation
\begin{equation}\label{gen:C}
tp^2x^2C^2(t,1)+px(C-pxC-p)C(t,1)-(C-pxC-p)=0.
\end{equation}
\end{lemma}
\begin{proof}
Let $\mathcal{B}_n$ be the set of Dyck paths in $\mathcal{D}_n$ that begin with an east step followed immediately by a north step. If we introduce 
$$
B(t,p;x):=\sum_{n\geq1}x^n\sum_{D\in\mathcal{B}_n}t^{\seg(D)}p^{\lseg(D)},
$$
then clearly 
\begin{equation}\label{gen:B}
B(t,p;x)=pxC(t,p;x).
\end{equation}

For $n\geq2$, a Dyck path $D=d_1\cdots d_n\in\mathcal{D}_n\setminus\mathcal{B}_n$  with $\min\{i\geq2: d_{i+1}=i\text{ or $i=n$}\}=j$ can be decomposed uniquely into a pair $(D_1,D_2)$  of Dyck paths, where $D_1=d_2d_3\cdots d_j\in\mathcal{D}_{j-1}$ and $D_2=(d_{j+1}-j)(d_{j+2}-j)\cdots(d_n-j)\in\mathcal{D}_{n-j}$ (possibly empty). This decomposition is reversible and satisfies the following properties: 
$$
\lseg(D)=
\begin{cases}
2,\quad&\text{if $D_1\in\mathcal{B}_{j-1}$}\\
1+\lseg(D_1), &\text{otherwise}
\end{cases}
$$
and
$$
\seg(D)=\seg(D_1)+\seg(D_2)+\chi(D_1\in\mathcal{B}_{j-1}).
$$
Turning this decomposition into generating functions yields 
\begin{equation}\label{gen:A}
C-B-p=tp^2xB(t,1)C(t,1)+px(C-B-p)C(t,1).
\end{equation}
Substituting~\eqref{gen:B} into~\eqref{gen:A} gives~\eqref{gen:C} after simplification. 
\end{proof}

We are ready to prove Theorem~\ref{thm:321}.

\begin{proof}[{\bf Proof of Theorem~\ref{thm:321}}] Setting $p=1$ in~\eqref{gen:C} and solving for $C(t,1)$ gives 
\begin{equation}\label{321des}
C(t,1)=\frac{1-\sqrt{-4tx^2+4x^2-4x+1}}{2x(tx-x+1)}.
\end{equation}
Substituting this into~\eqref{gen:C} and solving for $C=C(t,p;x)$ yields 
\begin{equation}\label{des:ldes}
C-\frac{p}{1-px}=\frac{tp^2(2tx^2-2x^2+2x-1+\sqrt{-4tx^2+4x^2-4x+1})}{(tx-x+1)(px-1)(2tx-2x+2-p+p\sqrt{-4tx^2+4x^2-4x+1})}.
\end{equation}
Let $\tilde{C}(p)=(C-\frac{p}{1-px})/p^2$.
Then  $\tilde{C}(p)$ is the  size generating function for $321$-avoiding permutations with at least one descent by the pair $(\des,\ldes-1)$.
For  $w\in\P_n^2(321)$, let 
$$
\plat(w):=|\{i\in[n-1]:w_i=w_{i+1}\}|
$$
be the number of {\em plateaux} of $w$. 
 Since each permutation form $w\in\P_n^2(321)$, whose permutation part $\pi$ is a $321$-avoiding permutation, can be obtained from $\pi$ by inserting some copies of $\neg1$ into the spaces not after the leftmost descent slot of $\pi$, we have 
\begin{align*}
f(t,q;x):=1+\sum_{n\geq1}x^n\sum_{w\in\P_n^2(321)} t^{\des(w)}q^{\plat(w)}
	=\Bigl(1+\frac{x}{1-qx}\Bigr)^2\tilde{C}(p)+\Bigl(1+\frac{x}{1-qx}\Bigr)\frac{1}{1-px},
\end{align*}
where we set $p=1+\frac{tx}{1-qx}$, and the case with $\pi=\id_n, n\ge 0$ is dealt with separately to form the second product. Combining this relationship with~\eqref{des:ldes} results in 
\begin{equation}\label{des:plat}
f(t,q;x)=\frac{(1-2tx^2-qx+2x^2-x+(qx-x-1)S)(1+x-qx)}{2x^2(tx-x+1)(q^2x-2qx+tx-q+2)},
\end{equation}
where 
$S:=\sqrt{1+4x(x-tx-1)}$.
Setting $q=1$ in~\eqref{des:plat} yields
$$
\frac{1-2tx^2+2x^2-2x-\sqrt{1+4x(x-tx-1)}}{2x^2(tx-x+1)^2},
$$
which proves the theorem after comparing with the size generating function for $321$-avoiding permutations by the number of peaks derived recently by Bukata et al. in~\cite[Thm.~3]{pud}. 
\end{proof}
\begin{remark}
The expression~\eqref{321des} was first proved by Barnabei et al.~\cite{bbs}. Our expression~\eqref{des:ldes} is a generalization of~\eqref{321des}. See also~\cite{kl} for a different generalization of~\eqref{321des}.  
\end{remark}

For  $w=w_1\cdots w_n\in\P_n^2(321)$, let $w^r=w_n\cdots w_1\in\P_n^2(123)$ be the {\em reversal} of $w$. Clearly, we have 
$$
\des(w^r)+1=n-\des(w)-\plat(w). 
$$
Thus, making the substitution $q\leftarrow t^{-1}$, $x\leftarrow tx$ and $t\leftarrow t^{-1}$ in~\eqref{des:plat}  gives the following generating function formula for counting $123$-avoiding $2$-arrangements in permutation form by $\des+1$. 
\begin{theorem}
We have the generating function formula
\begin{align}\label{123:2des}
1+\sum_{n\geq1}x^n\sum_{w\in\P_n^2(123)} t^{\des(w)+1} &=\frac{1-x-tx-2tx^2+2t^2x^2-(tx-x+1)T}{2tx^2(tx-x-1)(1-\frac{2t}{tx-x+1})},
\end{align}
where $T:=\sqrt{1+4tx(tx-x-1)}$.
\end{theorem}
The refinement $\sum_{w\in\P_n^2(123)} t^{\des(w)}$ of Catalan numbers appears to be new and the first few polynomials are
\begin{align*}
1,\quad2,\quad3+2t,\quad2+10t+2t^2, \quad2+12t+26t^2+2t^3,\quad 2+12t+56t^2+60t^3+2t^4. 
\end{align*}

\subsection{\texorpdfstring{Length-$3$ patterns for $1$-arrangements}{Length-3 patterns for 1-arrangements} in derangement form}

In general, the enumeration  of pattern avoiding $1$-arrangements in derangement form is harder than that in permutation form. Our computer program indicates that only one Wilf-equivalence exists for length-$3$ patterns on 
$1$-arrangements in derangement form. The enumerative sequences for the number of the other four Wilf-equivalence classes turn out to be new in OEIS (see Table~\ref{3-derange}). 

\begin{table}
\hrule

{\small
\begin{tabbing}
xxxxxxxxxxxx\=xxxxxxxxxxxxxxxxxxxxxxxxxxxxxxxxxxxxxxxx\= xxxxxxxxxxxxxxx\=xxxxxxxxxxxxx\= \kill

Pattern $p$\> First values of $|\D_n^1(p)|$: \> counted?\> in OEIS?\\
\\
$321$\> $ 1, 2, 5, 15, 48, 159, 538, 1850, 6446,22712,\ldots$ \> Algebraic g.f.   \> A289589?   \\
$132$\> Wilf-equivalent to pattern $321$ \> Algebraic g.f.  \> A289589?   \\
$231$\> $1,2,5,14,42,131,420,1376,4595,15573,\ldots$ \> open  \> new   \\
$123$\> $1, 2, 6, 19, 61, 202, 688, 2367,8316,29356,\ldots$\> open  \> new   \\
$312$\> $1, 2, 4, 10, 27, 78, 235, 736,2366,7772,\ldots$\> open  \> new    \\
$213$\>  $1, 2, 6, 19, 63, 210, 716, 2462,8604,30296,\ldots$ \> open  \> new 
\vspace{.1in}
\end{tabbing}
}

\hrule
\vspace{.2in}

\caption{Length-$3$ patterns  for $1$-arrangements in derangement form}
\label{3-derange}
\end{table}


We begin with  a refinement of  an intriguing  result due to Robertson,  Saracino and  Zeilberger~\cite{rsz} which asserts that $\fix$ has the same distribution over $\S_n(321)$ and $\S_n(132)$. Let $\pi\in\S_n$ be a permutation. Recall from~\eqref{def:ldes} that $\ldes(\pi)$ is the position of the  leftmost descent of $\pi$. Similarly, let 
$$
\rdes(\pi):=n-\max(\{i:\pi_i>\pi_{i+1}\}\cup\{0\})
$$
be the complement of the position of the {\em{\bf\em r}ightmost {\bf\em des}cent} of $\pi$. Let $\exc(\pi):=|\{i:\pi_i>i\}|$ and $\aexc(\pi):=|\{i:\pi_i<i\}|$ be the number of {\em excedances} and {\em anti-excedances} of $\pi$, respectively. Also, let  $\lmax(\pi)$ (resp.~$\rmin$) denote the number of {\em {\bf\em l}eft-to-right {\bf\em max}ima} (resp.~{\em {\bf\em r}ight-to-left {\bf\em min}ima} ) of $\pi$. It is clear that 
\begin{equation}\label{inverse}
(\exc,\lmax)\pi=(\aexc,\rmin)\pi^{-1}.
\end{equation}
\begin{lemma}\label{lem:knuth}
There exists a bijection $\mathcal{K}: \S_n(321)\rightarrow\S_n(132)$ such that for each $\pi\in\S_n$,
$$
(\fix,\exc,\ldes)\pi=(\fix,\aexc,\rdes)\mathcal{K}(\pi).
$$
\end{lemma}
\begin{proof}
The bijection $\mathcal{K}$ is a composition of a bijection due to Knuth and the inverse of permutations. For $\pi\in\S_n(321)$, let $\pi'$ be the $132$-avoiding permutation under Knuth's bijection (see the description in~\cite[Sec.~3.1]{ck} or~\cite{ep}). Then noting that a permutation $\sigma\in\S_n(132)$ if and only if $\sigma^{-1}\in\S_n(132)$, we define $\mathcal{K}(\pi)=(\pi')^{-1}$. It has been shown by Elizalde and Pak~\cite{ep} that $(\fix,\exc)\pi=(\fix,\exc)\pi'$ and by Claesson and  Kitaev~\cite[Sec.~6.4]{ck} that $\ldes(\pi)=\lmax(\pi')$. Thus, in view of~\eqref{inverse}, it remains to show that $\rmin(\sigma)=\rdes(\sigma)$ for each $\sigma\in\S_n(132)$. This follows from the observation that the suffix of $\sigma$ starting from the letter $1$ is monotonously increasing. 
\end{proof}

The next observation is obvious, but useful. 

\begin{Observation}\label{obser:der1}
Let $w\in\D_n^1$ be a word with weak derangement part $\Der(w)=\pi$. Then, $w$ is $321$-avoiding (resp.~$132$-avoiding) if and only if 
\begin{enumerate}
	\item $\pi$ is a $321$-avoiding (resp.~$132$-avoiding) derangement, and
	\item all copies of $\neg1$ appear not after  the leftmost  (resp.~not before the rightmost) descent slot of $\pi$.
\end{enumerate} 
\end{Observation}

We are ready to prove Theorem~\ref{der:wilf}.

\begin{proof}[{\bf Proof of Theorem~\ref{der:wilf}}] The first statement  that $\D_n^1(321)$ and $\D_n^1(132)$ have the same cardinality follows directly from Lemma~\ref{lem:knuth} and Observation~\ref{obser:der1}. 


Let $\mathcal{D}'_n$ be the set of Dyck paths $D\in\mathcal{D}_n$ such that $D$ has no hills.  In order to compute~\eqref{gen:321der}, we need to calculate the generating function
$$
D(p,x):=p+\sum_{n\geq1}x^n\sum_{\pi\in\D_n(321)}p^{\ldes(\pi)+1}=p+\sum_{n\geq1}x^n\sum_{D\in\mathcal{D}'_n}p^{\lseg(D)},
$$
where the second equality follows from Lemma~\ref{lem:kra}. Using the classical decomposition of Dyck paths as in the proof of Lemma~\ref{decom:dyck} we obtain
\begin{align}\label{0hill}
D(p,x) &= \frac{2p(1+2(1-p)x+\sqrt{1-4x})}{(1+2x+\sqrt{1-4x})(2-p+p\sqrt{1-4x})}.
\end{align}
By Observation~\ref{obser:der1} we have 
$$
 1+\sum_{n\geq1}|\D_n^1(321)|x^n=D\bigl(\frac{1}{1-x},x\bigr). 
$$
Combining this with~\eqref{0hill} we get~\eqref{gen:321der}, completing the proof. 
\end{proof}
\noindent
{\bf Acknowledgement.}
This work was supported
by the National Science Foundation of China grants 11871247 and  the project of Qilu Young Scholars of Shandong University.

\bigskip
\bigskip

\bibliographystyle{plain}


\end{document}